\documentclass[10pt]{amsart}
\usepackage{amsmath,amsfonts,amsthm,amssymb,amscd}
\def\classification#1{\def\@class{#1}}
\classification{\null}

\DeclareFontFamily{OT1}{rsfs}{}
\DeclareFontShape{OT1}{rsfs}{n}{it}{<-> rsfs10}{}
\DeclareMathAlphabet{\mathscr}{OT1}{rsfs}{n}{it}

\newcommand{\R}{{\mathbb R}}

\newcommand{\E}{\mathbb{E}}
\newcommand{\la}{\lambda}

\def\E{\mathsf {E}}

\newtheorem{conj}{Conjecture}

\newtheorem{theorem}{Theorem}

\newtheorem{lemma}[theorem]{Lemma}
\newtheorem{corollary}[theorem]{Corollary}
\theoremstyle{remark}

\def\E{\mathsf {E}}

\title{On discrete values of bilinear forms}

\author{Alex Iosevich, Oliver Roche-Newton, Misha Rudnev}
\address{Alex Iosevich, Department of Mathematics, University of Rochester, Rochester NY 14627, USA}
\email{iosevich@math.rochester.edu}

\address{Oliver Roche-Newton,  School of Mathematics and Statistics, Wuhan University, Wuhan,
Hubei Province, P.R.China 430072}
\email{o.rochenewton@gmail.com}

\address{Misha Rudnev, Department of Mathematics, University of Bristol,
  Bristol BS8 1TW, UK}
\email{m.rudnev@bristol.ac.uk}

\begin{document}
\begin{abstract} This paper is an erratum to  our paper, entitled{ \em On an application of Guth-Katz theorem.} Math. Res. Lett. {\bf 18} (2011), no. 4, 691--697,  \cite{IRR}.

Let $F$ be the real or complex field and $\omega$ a non-degenerate skew-symmetric bilinear form in the plane $F^2$. We prove that for finite a point set $P\subset F^2\setminus\{0\}$,  the set $T_\omega(P)$ of nonzero values of $\omega$ in $P\times P$, if nonempty, has cardinality $\Omega(N^{9/13}).$

A presumably near-sharp estimate $\Omega(N/\log N)$ was claimed in  \cite{IRR} over the reals for a symmetric or skew-symmetric  form $\omega$. However, the set-up for the proof was flawed. We discuss why we believe that justifying the claim of \cite{IRR} in full strength is a major open problem.

In the special case when $P=A\times A$, where $A$ is a set of at least two reals, we establish the following sum-product type estimates:
$$ |AA+ AA|= \Omega \left(|A|^{19/12}\right), $$ and
$$|AA-AA|= \Omega\left( \frac{|A|^{26/17}}{\log^{2/17}|A|}\right).$$

\end{abstract}

\maketitle

\section{Introduction}
Let $F$ be the real or complex field and let $\omega$ be a non-degenerate symmetric or skew-symmetric bilinear form on $F^2$. Consider an $N$-element point set $P\subset F^2\setminus\{0\}$, that is $|P|=N$. The question we ask is, what is the minimum cardinality of
\begin{equation}
T_\omega(P) =\{\omega(q,q'):\,q,q'\in P\}\setminus\{0\}.
\label{qu}\end{equation}
One exceptional case is that the set $T_\omega(P)$ can be empty, given that the form $\omega$ is skew-symmetric and $P$ is supported on a single line through the origin.   But what if one excludes this special case? 
The following  conjecture appears to be a central open question of discrete projective plane geometry. 

\begin{conj}\label{cj} $|T_\omega(P)|=\Omega^*(N),$  that is possibly modulo logarithmic terms in $N$. \end{conj}
In fact, we have no evidence that the logarithmic factors should be there. After establishing partial results towards the resolution of this problem in the main body of the paper, additional discussion of the conjecture is presented in the appendix. 

Throughout the rest of the paper the standard symbols $O$, $\Omega$, and $\Theta$ are used to imply absolute constants in upper and lower bounds, respectively, as well as the symbols $\ll$, $\gg,$ and $\approx$. $C$ and $c$ stand for positive absolute constants, which may change from line to line.

All the sets we are dealing with are assumed to be finite and have more than one element. For such a set $A\subset F\setminus\{0\}$ we use the standard notations for its sum set
$$
A+A = \{a_1+a_2:\,a_1,a_2\in A\},
$$
the difference set $A-A$, the product set $AA$, the ratio set $A:A$, proceeding in the same vein with notations for sets arising from other algebraic operations on finite sets.

We write $\E(A,B)$ for the additive energy of two sets $A,B \subseteq F$, that is
$$
    \E(A,B) = |\{ a_1+b_1 = a_2+b_2 ~:~ a_1,a_2 \in A,\, b_1,b_2 \in B \}|,
$$and just $\E(A)$ instead of $\E(A,A).$

%Similarly, $$\E_k (A) = |\{ a_1-a'_1 = \dots = a_k-a'_k ~:~ a_j,a'_j \in A \}|.$$

\medskip

Define $\E(P)$ as the number of quadruples $(q,q',r,r')$ satisfying the equation:
\begin{equation}\label{eng}
\omega(q,q')=\omega(r,r')\neq 0,\qquad q,\ldots,r'\in P.
\end{equation}

The bound $O(N^3\log N)$ on the number $\E(P)$, over the reals,  was claimed  in \cite{IRR} in the case when $\omega$ is the standard cross-product or the dot product. Similar to the construction in the celebrated paper of Guth-Katz  \cite{GK}, which settled the plane Erd\H os distinct distances problem,  it was claimed that a pair $(q,r)\in F^4$ defines a line in the group $SL_2(F)$, and solutions of \eqref{eng} are in one-to-one correspondence with lines' pairwise intersections. Unfortunately, we discovered an error in the set-up. 

Technically, the error   came down to ignoring the presence of possible degeneracy, inherent in the problem. If a quadruple $(q,q',r,r')$ satisfies equation \eqref{eng}, then so does $(\lambda q,q',\lambda r, r')$ for any nonzero $\lambda \in F$. This is clearly not the case if one deals with the corresponding equation
\begin{equation}\label{pd}\|q-q'\|=\|r-r'\|\neq 0,\qquad q,\ldots,r'\in P,\end{equation} for equal pairs of distances.

In other words, a line in question in $SL_2(F)$ is defined not by the pair $(q,r)\in F^4$, but $(q:r)\in FP^3$, that is as homogeneous coordinates in the projective three-space. The initial claim in \cite{IRR} that there arise $\Omega(N^2)$ distinct lines in three dimensions is false. It is to this set of lines the Guth-Katz incidence theorem was applied. Instead, one may have fewer lines with multiplicities, and this makes the estimates coming from the Guth-Katz theorem significantly worse. 

For example, in the case of $P=A\times A$, $A$ being a geometric progression, the number of distinct lines in three dimensions would be approximately $N^{3/2}=|A|^{3}$, rather than $N^2$, each line appearing with the multiplicity, or weight, of about $|A|$. The application of the Guth-Katz theorem would only yield the estimate $O(N^{13/4}\log N)$ on the number $\E(P)$, even losing a factor of $\log N$ versus what one gets after a single application of the Szemer\'edi-Trotter theorem, presented in the forthcoming Lemma \ref{bone}.

If one looks back at equation (6) in \cite{IRR}, it is clear that a line in three dimension it describes is defined not by the quadruple $(a,b,c,d)\in F^4$, but 
$(a:b:c:d)\in FP^3$.

\medskip
It was recently shown by the third author in \cite{MR}, that in the special case when the maximum weight is bounded as $O(1)$, one gets $|T_\omega(P)|\gg N$ over any field $F$, under a constraint $N\ll p$ in positive characteristic $p$. We do not address fields of positive characteristic further because we do not have a suitably strong analogue of the forthcoming Theorem \ref{crt}. The rest of the machinery used can be replaced on the basis of the results in \cite{MR}. For more discussion on sum-product type estimates in the positive characteristic case see \cite{RRS}.

\medskip
Qualitatively, the error in \cite{IRR} was the false presupposition that the main result was achievable by an adaptation of the proof of the renown Erd\H os distance problem, which had been beautifully resolved by Guth and Katz \cite{GK}. As the title of \cite{IRR} indicated, it came in the wake of the celebrated work \cite{GK}, but as we indicated above, the setup does not work in this context. It appears that the Erd\H os distance problem is very special. For one gets a sharp second moment bound, that is the upper bound for the number of congruent line segments defined by a Euclidean plane point set $P$ after a single application of an incidence theorem in the projective three-space $\mathbb RP^3$ that Guth and Katz succeeded in proving. The theorem had two counterparts, Theorems 2.10 and 2.11 in \cite{GK}, but we refer to their union as the Guth-Katz theorem. 

On the contrary, it was suggested much earlier by Chang \cite{Ch1} that the Erd\H os-Szemer\'edi sum-product conjecture, even in its weak form discussed in the appendix, can hardly be expected to get resolved merely by a ``black box" application of an incidence theorem, even though the Guth-Katz theorem was not yet available. The fact that the metric second moment estimate for distances in the plane can be converted to a projective one for lines intersecting in three dimensions was discovered by Elekes and Sharir \cite{ES} by looking at the $SE_2(F)$ symmetries that take one point of the plane point set to another. It was then shown by the third author and J. M. Selig \cite{RS} that what matters is the dimension and the algebraic structure of the equations stating that two line segments in the plane have the same length. For the space of line segments in the plane is four dimensional, just like the Klein quadric, the space of lines in the three-space. So \cite{RS} identified the whole family of two-dimensional metric problems, where the set of  pairs of congruent in the appropriate sense line segments with endpoints in a $N$-point set $P$ was put into one-to-one correspondence with a set of intersecting pairs of lines for a family of some $N^2$ lines in three dimensions.  

The Erd\H os distance problem is one of these problems, some others are analogous statements about spherical and hyperbolic distances. It was shown in \cite{RS} how the Guth-Katz theorem applies to these problems directly, bypassing any symmetry group considerations. Previously, the second and third authors \cite{RR} followed the Elekes-Sharir and Guth-Katz approach in order to show that $N$ points in the plane determine $\Omega(N/\log N)$ distinct Minkowski distances.

In summary, recent results in \cite{MR} and \cite{RS}, combined with the perspective of this paper establish a general framework for Erd\H os type problem in the plane over both finite and infinite fields. In this setup, the general problem of lower bounds for the number of distinct values of symmetric or skew-symmetric forms figures as the hardest of the remaining Erd\H os type problem in the plane, one that is most closely related to the celebrated sum-product conjecture. In this paper we take a small but non-trivial bite out of this problem.

\section{Main results}
The general result we present as an erratum is the following  theorem.

%\begin{theorem} \label{23} Let $P$ be a set of $N$ points in $F^2\setminus\{0\}$, with $N< p^{3/2}$ if $F$ has positive characteristic $p$. Let $\omega$ be a symmetric or skew-symmetric non-degenerate bilinear form, with $P$ not being supported on an isotropic  line through the origin. Then $|T_\omega(P)|\gg N^{2/3}$ .
%\end{theorem}
%We restate that since we work in any field $F$, $\omega$ always has an isotropic line through the origin in the algebraic %closure of $F$, that is such that $\omega(q,q)=0$ for any $q$ on that line.

\begin{theorem} \label{34} Let $P$ be a  set of $N$ points in $F^2\setminus\{0\}$, with $F$ being $\mathbb R$ or $\mathbb C$.  Suppose, the non-degenerate bilinear form $\omega$ is  skew-symmetric and $P$ is not  supported on a single line through the origin. Then
$$|T_\omega(P)|\gg N^{9/13}.$$
\end{theorem}

Unfortunately, this result is much weaker than the claim in \cite{IRR}, barely beating the lower bound $\Omega(N^{2/3})$ which comes from the (sharp) bound $O(N^{4/3})$ on the number of appearances of a single nonzero value of $\omega$, given by the Szemer\'{e}di-Trotter theorem.

\medskip
In the special case $P=A\times A$, over the reals, we establish the following.

%\begin{theorem} \label{epsilon1} Let $A\subset \C$ be finite. One has
%\begin{equation}\label{epsilon1eq} |AA-AA|\gg |A|^{3/2+\epsilon_1},\qquad |AA+AA|\gg |A|^{3/2+\epsilon_2},%\end{equation} for any $\epsilon_1<..$, $\epsilon_2<..$\end{theorem}

\begin{theorem} \label{epsilon1} Let $A\subset \R\setminus\{0\}$ be finite. Then
\begin{equation}\label{epsilon1eq} |AA+AA|\gg |A|^{5/4}|A:A|^{1/3}.\end{equation}
In particular
\begin{equation}\label{epsilon2eq} |AA+AA|\gg |A|^{19/12}.\end{equation}
 \end{theorem}

This result gives an improvement on the known bound
\begin{equation}
|AA\pm AA| \gg |A||A:A|^{1/2},
\label{basicST}
\end{equation}
which also follows from a simple application of the Szemer\'{e}di-Trotter Theorem\footnote{See Exercise 8.3.3. in \cite{TV}.}, provided that the ratio set $|A:A|$ is not too large.

To claim Theorem \ref{epsilon1} we take advantage of the recent state of the art sum-product estimates. These results rely crucially on the Szemer\'edi-Trotter theorem. Our proof builds on techniques introduced in the recent elaboration of Solymosi's approach to sum-products (see \cite{So} for the original construction and \cite{KR} for its adaptation to the complex case) by Konyagin and Shkredov \cite{KS}.  Balog \cite{B} was the first to obtain lower bounds for $|AA+AA|$ via this particular geometric approach.

The proof of Theorem \ref{epsilon1} does not extend to give a lower bound for $|AA-AA|$. Instead, we utilise the Szemer\'{e}di-Trotter Theorem along with the best known results for the sum-product problem for the case when $|A:A|$ is small, to record the following weaker claim.

\begin{theorem} \label{epsilon2} Let $A\subset \R\setminus\{0\}$ be finite. Then
\begin{equation}\label{epsilon3eq} |AA-AA|\gg \frac{|A|^{26/17}}{\log^{2/17}|A|}.\end{equation} \end{theorem}

%To claim Theorem \ref{epsilon1} we take advantage of the recent state of the art sum-product estimates of Shkredov and collaborators, see e.g. \cite{Sh}, \cite{KS}. These results rely on the Szemer\'edi-Trotter theorem in the plane, which is known in the strong enough form our purposes only over the real and complex field. They also rely on additive combinatorics tools, as well as Shkredov's machinery of higher degree convolutions. They also use the recent elaboration of Solymosi's approach to sum-products (see \cite{So} for the original construction and \cite{KR} for its adaptation to the complex case) by Konyagin and Shkredov, \cite{KS}.

%As was discussed in \cite{RRS}, we do not know whether a construction based on the forthcoming general incidence Theorem \ref{mish} would enable one to replace the use of the stronger  Szemer\'edi-Trotter theorem in order to beat the exponent $3/2$ for the size of $AA\pm AA$  for sufficiently small sets in the positive characteristic case.

\section{Preliminary results}

Our geometric argument uses the Szemer\'{e}di-Trotter Theorem. It is indispensable for the sum-product type estimates above. For the purposes of Theorem \ref{34} it can be replaced by a weaker theorem on point-plane incidences in three dimensions, see \cite{MR}.

\begin{theorem} \label{st} Let $P$ and $L$ be finite sets of points and lines respectively in the projective plane $FP^2$, where $F$ is the real or complex field.\footnote{The original proof of the Szemer\'edit-Trotter theorem over the reals appeared in \cite{S-T}. It was extended to the complex field by T\'oth in as early as in 2003, but the proof  came out in print only recently, \cite{T}.  For the special case of real/complex Cartesian product sets see a very short and elegant argument by Solymosi and Tardos \cite{ST}.} Then
$$I(P,L):=|\{(p,l) \in P \times L : p \in l \}| \: \ll \: |P|^{2/3}|L|^{2/3} +|P|+|L|.$$
\end{theorem}

It is convenient to record the following corollary.

\begin{corollary} \label{stcor} For arbitrary finite sets $A,B,C$ and $D$ of real or complex numbers, the number of solutions to the equation
\begin{equation}
a-b=cd,\,\,\,\,\,\,\,\,\,(a,b,c,d) \in A\times B \times C \times D,
\label{eqsp}
\end{equation}
is $O((|A||B||C||D|)^{2/3}+|A||D|+|B||C|)$.
\end{corollary}

\begin{proof} Let $l_{b,c}$ be the line with equation $y=cx+b$ and let $L=\{l_{b,c}:(b,c) \in B \times C \}$. Then, with $P=D \times A$, note that $I(P,L)$ is equal to the number of solutions to \eqref{eqsp}. An application of the Szemer\'{e}di-Trotter Theorem completes the proof.
\end{proof}

As a generalisation of Corollary \ref{stcor} let us quote a weighted version of the Szemer\'edi-Trotter theorem.  The set-up is a pair $(P, L)$ comprising a finite set of points and a finite set of lines in the projective plane $FP^2$, where $F$ is $\mathbb R$ or  $\mathbb C$. One has a positive real-valued weight function $w$ on $P\cup L$, with the supremum-norms $w_P,w_L$ and $L_1$-norms $W_P, W_L$ on the sets $P,L$, respectively. For a pair $(p,l)\in P\times L$ we write $\delta_{pl}$ as the characteristic function of the event that the point $p$ is incident to a line $l$. The number of weighted incidences $I_w$, defined below, is bounded as follows.

\begin{theorem}\label{wst} One has
\begin{equation}\label{wind}I_w:=\sum_{p\in P,\,l\in L}w(p)w(l)\delta_{pl}\;\ll \: (w_Pw_L)^{1/3} (W_PW_L)^{2/3}  + w_P W_L + w_L W_P.\end{equation}\end{theorem}
The statement follows after an easy weight rearrangement argument, followed by application of the Szemer\'edi-Trotter theorem, see e.g. \cite{IKRT}.

\medskip
The other component of the geometric argument in the proof of Theorem \ref{34} consists in  the following theorem, which deals with distinct cross-ratios, generated by a finite subset of the projective line $FP^1$. We recall that, for a quadruple of points $a,b,c,d\in FP^1$, their cross-ratio (also in $FP^1$) is defined as
$$r(a,b,c,d) = \frac{(a-b)(c-d)}{(a-c)(b-d)}.$$
Let $A\subset FP^1$, denote
$$
R(A) = \{r(a,b,c,d):\, a,b,c,d \in A\}
$$

\begin{theorem}  \label{crt} Let $A\subset FP^1$.
Then $|R(A)|=\Omega(|A|^2)$.\end{theorem}
Theorem \ref{crt} was proved in real and complex fields in the paper of Solymosi and Tardos \cite{ST} and later was re-discovered for reals by Jones \cite{J}. Note that the estimate is unlikely to be sharp: the correct one is probably $\Omega(|A|^3)$, possibly modulo logarithmic terms in $|A|$. One gets $O(|A|^3)$ distinct cross-ratios if $A$ is a geometric progression.

\section{Proof of Theorem \ref{34}}
In two dimensions a nontrivial skew-symmetric  form $\omega$ is given by the matrix $\left(\begin{array}{cc}0&1\\-1&0\end{array}\right)$, up to a multiplier, so we refer to its values further as {\em areas} and write simply $T(P)$ for the set of its nonzero values, defined by pairs of non-collinear vectors in $P$.

Let $P_1\subseteq P$ contain all those points of $P$, which are supported on lines through the origin with no more than some $w_0\geq 1$ points and $P_2=P\setminus P_1$. The parameter $w_0$ is to be chosen later. Let $N_1,N_2$ be the number of elements in $P_1,P_2$, respectively. Let $T_1,T_2$ stand for $T(P_1), T(P_2)$. We will obtain two different lower bounds for $T_1,T_2$.

The first of these becomes better for smaller $w_0$, whilst the second bound improves as $w_0$ increases. Balancing the two lower bounds leads to the desired result.
The first lower bound is the following. \begin{lemma}\label{bone}Assume that $P_1$ is not contained in a single line through the origin. Then
\begin{equation}
|T_1| = \Omega( N_1w_0^{-1/2} ).
\label{lb1}
\end{equation}
\end{lemma}

\begin{proof} The claim follows after an application of Theorem \ref{wst}. For each point in $p\in P_1$ and each $t\in T$ draw the line $\{q \in F^2:\omega(p,q)=t\}$. We obtain an arrangement of some weighted set $L$ of lines with the total weight $W_L=|T_1|N_1$ and maximum weight $w_0<N_1/2$ (or there is nothing to prove). Consider its incidences with the point set $P$. Applying Theorem \ref{wst} with $w_L=w_0$, $w_P=1$, $W_P=N_1$  we obtain
$$
I_w = O(w_0^{1/3} N_1^{4/3} |T_1|^{2/3}).
$$
Note that the last two terms in the estimate \eqref{wind} of Theorem \ref{wst} are dominated by the first one, or the bound \eqref{lb1} holds trivially.
On the other hand, we have $I_w=\Omega(N_1^2)$, since $I_{\omega}$ is equal to the number of non-collinear through the origin pairs of points in $P_1$. Estimate \eqref{lb1} follows.
\end{proof}

\medskip
Our second bound is as follows.
\begin{lemma}\label{btwo}
Suppose $P_2$ is supported on at least four distinct lines through the origin. Then \begin{equation}
|T_2| =\Omega( N_2^{3/7} w_0^{3/7}).
\label{lb2}
\end{equation}\end{lemma}

\begin{proof}
Consider the following equation
\begin{equation}\label{teq}
t_1t_2=t_3t_4-t_5t_6:\;t_1,\ldots,t_6\in T_2.
\end{equation}

The following lemma is central for the argument.
\begin{lemma} \label{key} Equation \eqref{teq} has $\Omega(N_2^2w_0^2)$ solutions.
\end{lemma}
Before proving Lemma \ref{key} let us show how it quickly implies \eqref{lb2}, and thus in turn Theorem \ref{34}. We apply Corollary \ref{stcor} (or equivalently one could apply Theorem \ref{wst}) to get an upper bound for the number of solutions of \eqref{teq}. Indeed, for any set $T$, the number of solutions to \eqref{teq} is bounded by
\begin{align*}|\{(t_1,\dots,t_6) \in T^6:t_1t_2=t_3t_4-t_5t_6\}| & \leq \sum_{t_3,t_5 \in T}|\{(t_1,t_2,t_4,t_6) \in T^6:t_1t_2=t_3t_4-t_5t_6\}|
\\& \ll |T|^2|T|^{8/3}=|T|^{14/3}.
\end{align*}
On the other hand, Lemma \ref{key} provides  the lower bound from  for the same quantity. Comparing the two yields \eqref{lb2}.

This suffices to wrap up with the proof of Theorem \ref{34}. We set $w_0 = N^{8/13}$ and observe that one of $N_1,N_2$ is at least $N/2$, and if the corresponding assumption that $P_1$ is not supported on a single line through the origin or $P_2$ on fewer than four lines through the origin were false, there would be nothing to prove.
\end{proof}

\begin{proof}[Proof of Lemma \ref{key}]

 Let  $(a,b,c,d)$ be a quadruple of points in $P_2$, lying in four distinct directions from the origin. Denote $t_{ab}=a_1b_2-a_2b_1$ the signed area of a triangle $Oab$, where $a=(a_1,a_2)$, $b=(b_1,b_2)$ and similarly for other pairs of points  in the quadruple.

We have the identity
\begin{equation}\label{bac}
t_{ad}t_{cb} =  t_{ab}t_{cd} - t_{ac}t_{bd},
  \end{equation}
 thus getting a particular instance of equation \eqref{teq}.

There are many ways to verify \eqref{bac}: one can do it, e.g.  by direct calculation using coordinates, trigonometry, or  using the cross and scalar product notation from the fact that
 $$
(a\times d)\cdot (b\times c) = -c\cdot(b\times(a\times d)) = (a\cdot b)(c\cdot d) - (a\cdot c)(b\cdot d),
 $$
 now replace $b=(b_1,b_2)$ by $b^\perp=(-b_2,b_1)$, the same for $c$.

Consider a quadruple of oriented areas $(t_{ab},t_{cd},t_{ac},t_{bd})$. Let us introduce the equivalence relation
$(a,b,c,d)\sim(a',b',c',d')$ if all the corresponding {\em four} areas $(t_{ab},t_{cd},t_{ac},t_{bd})=(t_{a'b'},t_{c'd'},t_{a'c'},t_{b'd'})$ are the same. Note that if $(a,b,c,d)$ and $(a',b',c',d')$ come from different equivalence classes, then the solutions they give to \eqref{bac} are distinct when viewed as solutions to \eqref{teq}.

Clearly, $(a,b,c,d)\sim(a',b',c',d')$ only if $\frac{t_{ab}t_{cd}}{t_{ac}t_{bd}} = \frac{t_{a'b'}t_{c'd'}}{t_{a'c'}t_{b'd'}}$. But the latter quantity is the cross-ratio\footnote{The cross-ratio of four directions $(\delta_a,\delta_b,\delta_c,\delta_d)$ from the origin is defined as follows. Take a line $l$ not passing through the origin. Let $a,b,c,d$ points of intersection of the latter line with the lines in the directions $(\delta_a,\delta_b,\delta_c,\delta_d)$, respectively. Viewing $l$ as $FP^1$, set $r=\frac{(a-b)(c-d)}{(a-c)(b-d)}$. Now observe that $r$ is the ratio of the corresponding triangle areas rooted at the origin and that viewed this way, the points $a,b,c,d$ can move along the corresponding lines in the directions $(\delta_a,\delta_b,\delta_c,\delta_d)$ without changing the value of $r$. See, e.g. \cite{SK}.}. That is, if $\delta_a,\ldots,\delta_{d'}\in FP^1$ are  the directions from the origin, corresponding to the points $a,\ldots,d'$, respectively, then $\frac{t_{ab}t_{cd}}{t_{ac}t_{bd}}= r(\delta_a,\delta_b,\delta_c,\delta_d)$, with the same equation holding for $a',b',c',d'$, and thus $r(\delta_a,\delta_b,\delta_c,\delta_d) = r(\delta_{a'},\delta_{b'},\delta_{c'},\delta_{d'})$.

By Theorem \ref{crt} there are $\Omega(|D|^2)$ distinct cross-ratios, where $D$ is the set of directions defined by $P_2$. Besides, fixing a representative
$(\delta_a,\delta_b,\delta_c,\delta_d)$ quadruple of directions with a given cross-ratio, choosing any $a,b,c,d$  in $P_2$ in the given directions yields different solutions of \eqref{teq}. Indeed, the right-hand side of \eqref{bac} will remain the same only if $a,d$ are dilated by the same factor, while $b,c$ are contracted by the same factor. But this clearly changes the left-hand side.

It remains to show that we can basically assume that $|D| =|P_2|/w_0$, for if we have fewer directions things get better. The easiest way of doing it is using induction in $|D|$.  Let $c<1$ be the universal constant implicit in Theorem \ref{crt}. Take, say $c'=c/16.$  The lemma is trivially true, with the hidden constant $c'$, if $|D|=4$, for we can certainly assume that no line through the origin supports $\Omega(N_2)$ points.

Now partition $P_2$ into $P_2'$ and $P_2''$, where $P_2'$ corresponds to points on lines through the origin, supporting between $w_0$ and $2w_0$ points, $P_2''$ being the rest of $P_2$. Then if $P_2''$ contains at least half of $P_2$, we are done by the induction assumption, for $P_2$ alone will yield at least $c'(|P_2|^2/4) (2w_0)^2$ solutions to \eqref{teq}.
 
Otherwise $P_2'$ contains at least half of $P_2$. It determines at least $|P_2|/(4w_0)$ directions. Hence there are at least $c |P_2|^2/(4w_0)^2$ distinct cross-ratios and $c (|P_2|^2w_0^2)/16=c'|P_2|^2w_0^2$ solutions to \eqref{teq}.

\end{proof}

\section{Proof of Theorem \ref{epsilon1}}

Let the finite set of reals $A$ have more than one element. Define
$$d(A)=\min_{C \neq \emptyset} \frac{|AC|^2}{|A||C|}.$$
To prove Theorem \ref{epsilon1} we will need the following result, which is \cite[Corollary 8]{KS}.
\begin{lemma} \label{thm:ST}
Let $A$ be a finite sets of reals and $\alpha_1,\alpha_2$ and $\alpha_3$ non-zero real numbers. Then the number of solutions to the equation
$$\alpha_1a_1+\alpha_2a_2+\alpha_3a_3=0,$$
such that $a_1,a_2 ,a_3 \in A$, is at most
$$C\cdot d^{1/3}(A)\cdot |A|^{5/3},$$ for some absolute constant $C$.
\end{lemma}

Lemma \ref{thm:ST} is a consequence of the Szemer\'{e}di-Trotter theorem, building on the work of Li and Roche-Newton \cite{LR} and Schoen and Shkredov \cite{SS}.

\begin{proof}

Without loss of generality, we may assume that $A$ consists of strictly positive real numbers. Following the notation of \cite{KS}, for a real nonzero $\lambda$, define
$$ \mathcal A_{\lambda}:= \{(x,y) \in A \times A : \frac{y}{x}=\lambda\},$$
and its projection onto the horizontal axis,
$$A_{\lambda}:=\{x:(x,y) \in \mathcal A_{\lambda}\}.$$
Note that $|A_{\la}|=|A \cap \la A|$ and
\begin{equation}
\sum_{\la} |A_{\la}|=|A|^2.
\label{obvious}
\end{equation}

For each $\lambda \in A:A$, we identify an arbitrary element from $\mathcal A_{\la}$, which we label $(a_{\lambda},\lambda a_{\lambda})$. Then, fixing distinct two slopes $\lambda_1$ and $\lambda_2$ from $A:A$ and following the observation of Balog \cite{B}, we note that at least $|A|^2$ distinct elements of $(AA+AA) \times (AA+AA)$ are obtained by summing pairs of vectors from the two lines through origin with slope $\lambda_1$ and $\lambda_2$. Indeed,
$$A(a_{\lambda_1},\lambda_1 a_{\lambda_1})  + A (a_{\lambda_2},\lambda_2a_{\lambda_2})  \subset (AA+AA) \times (AA+AA),$$
where for $\lambda \in A:A$
$$
A(a_{\lambda},\lambda a_{\lambda}) = \{(aa_{\lambda},\lambda aa_{\lambda}):\,a\in A\}.
$$
Note that these $|A|^2$ vector sums have slope in between $\lambda_1$ and $\lambda_2$. This is a consequence of the observation of Solymosi \cite{So} that the sum set of $m$ points on one line through the origin and $n$ points on another line through the origin consists of $mn$ points lying in between the two lines.  This fact expresses linear independence of two vectors in the two given directions, combined with the fact that multiplication by positive numbers preserves order of reals. Hence the assumption that the points lie inside the positive quadrant of the plane.

Following the strategy of Konyagin and Shkredov \cite{KS}, we split the family of $|A:A|$ slopes into clusters of $M$ consecutive slopes, where $2\leq M \leq |A:A|$ is a parameter to be specified later. The idea is to show that each cluster determines many different elements of $(AA+AA) \times (AA+AA)$. Since the slopes of these elements are in between the maximal and minimal values in that cluster, we can then sum over all clusters without overcounting.

If a cluster contains exactly $M$ lines, then it is called a \textit{full cluster}. Note that there are $\left\lfloor \frac{|A:A|}{M} \right\rfloor \geq \frac{|A:A|}{2M}$ full clusters, since we place exactly $M$ lines in each cluster, with the possible exception of the last cluster which contains at most $M$ lines.

Let $U$ be a full cluster, with shallowest slope $\lambda_{min}$ and steepest slope $\lambda_{max}$. Let $\mu(U)$ denote the number of elements of $(AA+AA) \times (AA+AA)$ which lie in between $\lambda_{min}$ and $\lambda_{max}$. Then, by the inclusion-exclusion principle

%Let $k=|\mathcal{L}_j|$ and write $\mathcal L_j=\{l_1,l_2,\dots,l_{k}\}$, where the slope of $l_i$ is smaller than that of $l_{i+1}$. We denote the slope of the line %$l_i$ by $\lambda_{l_i}$. The set of lines $\mathcal L_j$ is now divided into clusters of size $M$, with each cluster consisting of consecutive lines from $\mathcal %L_j$. So, the first cluster is $U_1=\{l_1,l_2,\dots,l_M\}$, and a cluster $U_i$ consists of $M$ consecutive lines, with the possible exception of the last cluster. %However, there are at least $\lfloor \frac{k}{M} \rfloor \geq \frac{k}{2M}$ clusters containing $M$ lines.

%Fix a cluster of lines $U_i=\{l_{(i-1)M+1},\dots l_{iM}\}$ of size $M$. The aim is to show that by taking vector sums of pairs of lines from $U_i$, we obtain many %distinct elements of $(AA+A) \times (AA+A)$, all of which have a slope which is in between $\lambda_{l_{(i-1)M+1}}$ and $\lambda_{l_{iM}}$. We will then sum %over all clusters.

\begin{equation}\mu(U) \geq |A|^2 {M \choose 2} - \sum_{\lambda_1,\lambda_2,\lambda_3,\lambda_4 \in U_j: \{\lambda_1,\lambda_2\} \neq \{\lambda_3,\lambda_4\}} \E(\lambda_1,\lambda_2,\lambda_3,\lambda_4),
\label{mucount2}
\end{equation}
where
$$
\E(\lambda_1,\lambda_2,\lambda_3,\lambda_4)  :=
 |\{[A(a_{\la_1},\la_1 a_{\la_1}) +A(a_{\lambda_2},\lambda_2a_{\lambda_2})]\cap [A(a_{\lambda_3},\lambda_3a_{\lambda_3})+A(a_{\lambda_4},\lambda_4a_{\lambda_4})] \}|.
$$

The next task is to obtain an upper bound for $\E(\lambda_1,\lambda_2,\lambda_3,\lambda_4)$ for an arbitrary quadruple $(\la_1,\la_2,\la_3,\la_4)$ which satisfies the aforementioned conditions.

Suppose that
$$
z=(x,y) \in [A(a_{\lambda_1},\lambda_1a_{\lambda_1})+A(a_{\lambda_2},\lambda_2a_{\lambda_2})]
\cap  [A(a_{\lambda_3},\lambda_3a_{\lambda_3})+A(a_{\lambda_4},\lambda_4a_{\lambda_4})],
$$
that is
$$
(x,y) =(aa_{\la_1},a\lambda_1a_{\la_1})+(ba_{\lambda_2},b\lambda_2a_{\lambda_2})
=(ca_{\la_3},c\lambda_3a_{\la_3})+(da_{\lambda_4},d\lambda_4a_{\lambda_4}),
$$
for some $a,b,c,d \in A$. Therefore,
$$\begin{array}{lccccccc}
x&=&aa_{\la_1}+ba_{\lambda_2}&=&ca_{\la_3}+da_{\lambda_4}\\
y&=&a\lambda_1a_{\la_1}+b\lambda_2a_{\lambda_2}&=&c\lambda_3a_{\la_3}+d\lambda_4a_{\lambda_4}.
\end{array}$$
It follows from the conditions on the quadruple $(\la_1,\la_2,\la_3,\la_4)$ that at least one of its members differs from the other three. Without loss of generality $\la_4\neq \la_1,\la_2,\la_3$. Then
$$0=a\lambda_1a_{\la_1}+b\lambda_2a_{\lambda_2}-c\lambda_3a_{\la_3}-d\lambda_4a_{\lambda_4} - \lambda_4(aa_{\la_1}+ba_{\lambda_2}-ca_{\la_3}-da_{\lambda_4}),$$
and thus
\begin{equation}
0=aa_{\la_1}(\lambda_1-\lambda_4)+ba_{\lambda_2}(\lambda_2-\lambda_4)+ca_{\la_3}(\lambda_4-\lambda_3).
\label{STsetup}
\end{equation}

Note that the values $a_{\la_1}(\lambda_1-\lambda_4), a_{\lambda_2}(\lambda_2-\lambda_4)$ and $a_{\la_3}(\lambda_4-\lambda_3)$ are all non-zero. We have shown that each contribution to $\E (\lambda_1,\lambda_2,\lambda_3,\lambda_4)$ determines a solution to \eqref{STsetup}. Furthermore, the solution $(a,b,c)$ to \eqref{STsetup} that we obtain via this deduction is unique. That is, if we start out with a different element
$$z \in [A(a_{\lambda_1},\lambda_1a_{\lambda_1})+A(a_{\lambda_2},\lambda_2a_{\lambda_2})]
\cap  [A(a_{\lambda_3},\lambda_3a_{\lambda_3})+A(a_{\lambda_4},\lambda_4a_{\lambda_4})],$$
we obtain a different solution to \eqref{STsetup}. It therefore follows from an application of Lemma \ref{thm:ST} that
$$\E(\lambda_1,\lambda_2,\lambda_3,\lambda_4) \leq C\cdot d^{1/3}(A)|A|^{5/3},$$
where $C$ is an absolute constant. Therefore,
\begin{align}
\mu(U)&\geq {M \choose 2} |A|^2 - M^4Cd^{1/3}(A)|A|^{5/3}
\\& \geq \frac{M^2}{4}|A|^2 - M^4Cd^{1/3}(A)|A|^{5/3}.
\label{mu}
\end{align}
Choosing the integer valued parameter
\begin{equation}M:=\left \lfloor \frac{|A|^{1/6}}{\sqrt{8C}d^{1/6}(A)} \right \rfloor\label{chM}\end{equation}
yields
\begin{equation}
\mu(U) \geq \frac{M^2}{8}|A|^2.
\label{mu2}
\end{equation}
Recall that, in fact we need $2 \leq M \leq |A:A|$. It is easy to check that the upper bound for $M$ is satisfied. Besides,  $M\leq 2$ implies that $d(A)\gg |A|$, in particular $|A:A|\gg |A|^{3/2}.$ Then the basic claim \eqref{basicST} results in a  stronger statement than Theorem \ref{epsilon1} and we are done.

Thus we proceed assuming that $2 \leq M \leq |A:A|$. Summing over the full clusters, of which there are at least $\frac{|A:A|}{2M}$, yields
\begin{align}
|AA+AA|^2\geq \frac{|A:A|}{2M}\frac{M^2}{8}|A|^2
\gg \frac{|A:A||A|^{13/6}}{d^{1/6}(A)}.
\label{aa+a}
\end{align}

To complete the proof, we need a suitable upper bound for $d(A)$. We simply use the trivial\footnote{Another approach would be to use non-trivial arguments from \cite{Sh2} to prove the bound $d(A) \ll |AA+AA|^2/|A|^3$, which implies that $|AA+AA| \gg |A|^{8/7}|A:A|^{3/7}$. This result is better than Theorem \ref{epsilon1} when $|A:A|$ is large. However, since this does not result in better unconditional bounds for $|AA+AA|$, we do not pursue the details here.} bound $d(A) \leq |A:A|^2/|A|^2$. Combining this inequality with \eqref{aa+a} and rearranging, we get
$$|AA+AA| \gg |A:A|^{1/3}|A|^{5/4},$$
which completes the proof of \eqref{epsilon1eq}. To prove \eqref{epsilon2eq}, we simply apply the trivial bound $|A:A| \geq |A|$.
\end{proof}

\begin{proof}[Proof of Theorem \ref{epsilon2}]
It was established in \cite[Theorem 11]{KS} that
\begin{equation}
|A:A|^{6}|A-A|^{5} \gg \frac{|A|^{14}}{\log^2|A|}.
\label{LiORN}
\end{equation}
Plugging this into the bound \eqref{basicST} and rearranging, we obtain
$$|AA-AA|^{12} \gg \frac{|A|^{26}}{|A-A|^{5}\log^{2}|A|}.$$
The claim of Theorem \ref{epsilon2} follows by applying the trivial bound $|A-A| \leq |AA-AA|$ and rearranging.
\end{proof}

\section*{Appendix}
In this section we give two heuristic arguments, showing the relation of Conjecture \ref{cj} to other reputable topics in geometric combinatorics and incidence theory.

\subsection*{Conjecture \ref{cj} and the Erd\H os-Szemer\'edi problem}
Let us call the {\em weak Erd\H os-Szemer\'edi conjecture} the statement that for $A\subset F$, the assumption $|AA|=O(|A|^{1+\delta})$ implies that $|A+A|=\Omega(|A|^{2-\epsilon})$ where the small parameters $(\delta,\epsilon)$ are on a polynomial curve through the origin. The question has been resolved qualitatively in the integer case, and therefore the $F=\mathbb Q$ case by Bourgain and Chang \cite{BC}, the dependence between $\delta$ and $\epsilon$ not having been made explicit. It is open over the real, complex, and $p$-adic fields. 

Let us show that a slightly stronger statement than Conjecture \ref{cj} implies the weak Erd\H os-Szemer\'edi conjecture. Namely suppose, we have the second moment bound $O(|P|^3)$ on the number $\E(P)$ of solutions of equation \eqref{eng}.

Then we have, assuming $0\not\in A$:
$$\begin{aligned}
\E(A) & =  |\{(a_1,a_2,a_3,a_4)\in A\times A\times A\times A:\, a_1+a_2=a_3+a_4\}| \\
& =  |A|^{-4}|\{(a_1,\ldots, a_8)\in A\times \ldots \times A:\,  a_1a_5/a_5+a_2a_6/a_6=a_3a_7/a_7+a_4a_8/a_8\}|\\
& \leq |A|^{-4}|\{(p_1,p_2,p_3,p_4, b_1,b_2,b_3,b_4)\in AA\times \ldots \times AA \times A^{-1}\times\ldots\times A^{-1}:\, \\& \hspace{16mm} p_1b_1+p_2b_2=p_3b_3+p_4b_4\}| \\
& \ll |AA|^3|A|^{-1},
\end{aligned}$$
after applying the assumption for the plane point set $AA\times A^{-1}$.

It follows, by the Cauchy-Schwarz inequality that
$$
|A\pm A| \geq \frac{|A|^4}{\E(A)} \gg \frac{|A|^5}{|AA|^3},
$$
thus if $|AA|< |A|^{1+\delta},$ then $|A\pm A|\gg |A|^{2-3\delta}.$

\subsection*{Example: on popular values of $\omega$}
We conclude this note by revisiting the well-known Erd\H os construction showing that one can have a plane point set $P$ of  $\Theta(N)$ points with $\gg 1$ nonzero values of the form $\omega$ repeating $\Omega(N^{4/3})$ times. By the Szemer\'edi-Trotter theorem a single nonzero value of $\omega$ may only have $O(N^{4/3})$ realisations.

Without loss of generality let $\omega$ be the standard dot product. The equation $q'\cdot q =1$ can be viewed as an incidence relation between $N$ points described by vectors $q\in P$  and $N$ lines described by covectors $q'$. There are well known constructions, similar to the forthcoming one, where the number of incidences is $\Omega(N^{4/3}).$

Let us also show that the pinned version of equation \eqref{eng}
\begin{equation}\label{hm}
 \omega(q,q')=\omega(q,r')\neq 0:\; q,q',r'\in P
\end{equation}
may have as many as $\Omega(N^{7/3})$ solutions. Therefore, it is impossible to use it without additional assumptions on $P$ to settle Conjecture \ref{cj}.

Contrary to this,  we do not have evidence that equation \eqref{eng} can possibly have more than $O(N^{3})$ solutions. On the other hand,  we cannot prove at the moment an upper bound better than $O(N^{10/3})$ for the number of solutions of \eqref{eng}. The latter bound follows, e.g.,  from the $O(N^{4/3})$  bound on the number of realisations of a single nonzero value of $\omega$.

\medskip

To avoid taking integer parts, take $N$ to be $6$th power of an even integer. Let $P$ be the union of the integer grid square $P_1= [-\sqrt{N},\ldots,\sqrt{N}] \times [-\sqrt{N},\ldots, \sqrt{N}]$ and the constructed below set $P_2$ of approximately $N$ points, with some $N^{2/3}$ points lying in each of the approximately $N^{1/3}$ directions from the origin of the set of lines $L$, defined as follows. 

Take all the points with co-prime coordinates $(a,b)$ in the sub-square  $[-\sqrt[6]{N},\ldots, \sqrt[6]{N}] \times [-\sqrt[6]{N},\ldots, \sqrt[6]{N}]$ and define a family $L$ of $\Theta(N^{1/3})$ lines through the origin by the following equations:
$$
bx-ay=0,\qquad |a|,|b|\leq \sqrt[6]{N}:\;\mbox{ gcd}(a,b)=1.
$$
Clearly, $L=L^\perp$. Each line in $L$ supports  $\Theta(N^{1/2-1/6=1/3})$ points of $P_1$.

Now translate $L$ from the origin to every point of the subset  $[-\sqrt{N}/2,\ldots, \sqrt{N}/2] \times [-\sqrt{N}/2,\ldots, \sqrt{N}/2]$ of $P_1$. Since each line in $L$ supports $\Theta(N^{1/3})$ points of $P_1$, in the union $\mathcal L$ of the translated lines one gets $\Theta(N)$ lines in $\Theta(N^{1/3})$ directions, hence  $\Theta(N^{2/3})$ parallel lines in each direction. (Each line in $L$ gets translated $\Theta(N)$ times but two translates of $l\in L$ result in the same line in $\mathcal L$ if their difference is in $P_1\cap l$.)
In particular, since $L=L^\perp$ for every line in $L$, there are  $\Theta(N^{2/3})$ lines in $\mathcal L$, perpendicular to it, each supporting $\Theta(N^{1/3})$ points of $P_1$.

The equations of lines in $\mathcal L\setminus L$ are
$$
bx-ay=bi-aj\neq 0,\qquad |a|,|b|\leq \sqrt[6]{N}:\;\mbox{ gcd}(a,b)=1,\qquad |i|,|j|\leq \sqrt{N}/2.
$$
By construction, for every $(a,b)$ there are $\Theta(N^{2/3}$) values of the right-hand side above. Also, there are $\Theta(N)$ distinct points in the set
$$
P_2:= \left\{\left(\frac{b}{bi-aj},\,\frac{a}{bi-aj}\right): \,   |a|,|b|\leq \sqrt[6]{N},\,\mbox{ gcd}(a,b)=1;\; |i|,|j|\leq \sqrt{N}/2,\; bi-aj\neq 0 \right\}.
$$
Besides,  the equation $q\cdot q'=1$, with $q\in P_1,\,q'\in P_2$ has $\Omega(N^{4/3})$ solutions. The same can be said about the equation $q\cdot q'=c$, for at least a $O(1)$ set of values of $c$, by scaling $P_1$. So, one has $\Omega(N^{4/3})$ realisations for each of $\gg 1$ distinct dot product values on the set $P=P_1\cup P_2$ of $\Theta(N)$ points.

Observe that the set $P_2$ has the property of being the union of sets of $\Theta(N^{2/3})$ points supported on each of the $\Theta(N^{1/3})$ lines in $L$.  In fact, for the purposes of having $\Omega(N^{7/3})$ solutions to equation \eqref{hm} it suffices to take {\em any} set $P_2$ with the above property. Indeed, consider  equation \eqref{hm} with $q\in P_2$ and $q',r'\in P_1$. Each nonzero vector $q$ has $\Theta(N^{2/3})$ lines  in $\mathcal L$, perpendicular to it, each line supporting approximately $N^{1/3}$ points of $P_1$.

Thus  the number of solutions of equation \eqref{hm}, that is the number of solutions of $q\cdot(q'-r')=0$ is at least a constant times
\begin{equation}\label{imp}
N\cdot N^{2/3} \cdot (N^{1/3})^2 = N^{7/3}.
\end{equation}

This example suggests that there is a principal difference between Conjecture \ref{cj} -- which has been shown to be related to the Erd\H os-Szemer\'edi sum-product conjecture -- and the Erd\H os distance problem. For the analogue of \eqref{hm} set $q=r$ in \eqref{pd}. However, an analogue of the above construction leading to the estimate \eqref{imp} is impossible, for Pach and Tardos \cite{PT} proved the bound $O(N^{2.137})$ for the number of isosceles triangles defined by a point set in $\R^2$. Moreover, the Erd\H os distance type questions  would all follow up to $o(1)$ exponents from the single distance conjecture, claiming that any distance in the Euclidean plane point set repeats  only $O(N^{1+o(1)})$ times  (see, e.g., the book by Brass, Moser, and Pach \cite{BMP} for details). This is also not the case with single nonzero values of $\omega$, whose set appears to be considerably more volatile. 

We finish by noticing that an analogue of the construction in this section does not appear to be feasible within a single Cartesian product set $P=A\times A$, where the problem of distinct values of bilinear forms meets the Erd\H os-Szemer\'edi conjecture. 

\section{Acknowledgements}
We are grateful to Ernie Croot and Ilya Shkredov for helpful discussions, and for their -- as well as other colleagues' -- patience while this erratum was being drafted.

\end{document}